\documentclass[10pt]{article}
	  \usepackage{lineno}
	  \usepackage{color}
\usepackage{cite}
\usepackage{color}
\usepackage{amsmath}
\usepackage{amsthm}
\usepackage{amsfonts}
\usepackage{amssymb}
\usepackage{bbm}
\usepackage{graphicx}
\usepackage{tikz}

\usetikzlibrary{math}

\newcommand{\R}{\mathbb{R}}

\newcommand{\Z}{\mathbb{Z}}
\newcommand{\N}{\mathbb{N}}

\DeclareMathOperator{\lcm}{lcm}

\newtheorem{theorem}{Theorem}[section]
\newtheorem{corollary}{Corollary}[theorem]
\newtheorem{lemma}[theorem]{Lemma}

\title{Compositions with 3 Pairwise Coprime Parts\footnote{This work will also appear in  the author's
doctoral thesis at Drexel University.}}

\author{James Thomas}
\date{}

\begin{document}
	\maketitle
	\begin{abstract}
		How many ways can we write $n$ as a sum of $3$ positive integers, no pair of which share a common factor? We express this quantity in terms of the number of solutions to a certain class of linear Diophantine equations. This allows us to show that there are
			$$
			\prod_{p \mid n} \left( 1- \frac{1}{p^2} \right) \prod_{q \nmid n} \left( 1- \frac{3}{q^2} \right) \frac{n^2}{2} + O(n^{3/2+o(1)})
			$$
		such compositions, where the products are over primes that respectively do and don't divide $n$. This strengthens the previous result of Bubbolini, Luca, and Spiga \cite{Luca}
	\end{abstract}

\section{Introduction}
	How many ways can we write $n$ as the sum of $k$ positive integers, no pair of which has a common factor? Formally, let $\textbf{T}_k(n)$ be the set of $k$-compositions $t=(t_1,...,t_k)$ such that
		\begin{align*}
		t_1 + ... t_k &= n
		\\
		\gcd(t_i,t_j) &= 1 \quad \forall i\not= j
		\end{align*}
	Let $T_k(n) := |\textbf{T}_k(n)|$. This problem can be solved exactly for the case $k=2$, where
		\begin{align*}
		\textbf{T}_2(n) &= \{(t,n-t) \ \mid  \  \gcd(t,n) = 1, \ 1\leq t \leq n-1 \}
		\\
		T_2(n) &= \varphi(n) = n\prod_{p\mid n}\left( 1 - 1/p \right)
		\end{align*}
	where $\varphi(n)$ is the Euler totient function. For $k\geq 3$ no exact formula is known although asymptotic results exist. The authors of \cite{Luca} showed that there is an arithmetic function  $f_k(n)$  such that if $n\geq e^{k2^{k+2}}$ then
		\begin{equation} \label{lucaresult}
		\Big|T_k(n) - f_k(n) \frac{n^{k-1}}{(k-1)!}\Big| \leq \frac{707n^{k-1}}{\log n}
		\end{equation}
	In this paper we sharpen the bound in the case $k=3$ using a new method. For the rest of this paper, let $\textbf{T}:= \textbf{T}_3(n)$. If $\textbf{X}$ is a set then let $X:=|\textbf{X}|$. Also let
		$$
		f(n) := f_3(n) = \prod_{p \mid n} \left( 1- \frac{1}{p^2} \right) \prod_{q \nmid n} \left( 1- \frac{3}{q^2} \right) 
		$$
	Our main result is
		\begin{theorem}\label{maintheorem}
		\begin{equation*}
		T = f(n)\frac{n^2}{2} + O(n^{3/2+o(1)})
		\end{equation*}
		\end{theorem}
	We'll give a brief overview of our proof. Let $\textbf{S}$ be the set of compositions $s=(s_1,s_2,s_3)$ of $n$ with $3$ parts. For each prime $p$ and pair of indices $i,j$ we can ask whether $p$ divides $\gcd(s_i,s_j)$. If the answer is ever yes, then $s$ is not in $\textbf{T}$. Conversely if the answer is always no, then $s\in \textbf{T}$. 

	For  a triple of positive integers $a = (a_1,a_2,a_3)$ let $\textbf{S}(a) \subseteq \textbf{S}$ be the set of compositions $s$ such that $a_i \mid s_i$ for all $i$. In particular, we have $\textbf{S} = \textbf{S}(1,1,1)$. Then
		\begin{equation}\label{pie}
		\textbf{S} \setminus \textbf{T}= \bigcup_{p\leq n}\left(\textbf{S}(p,p,1)\cup \textbf{S}(p,1,p) \cup \textbf{S}(1,p,p)\right)
		\end{equation}
	Where the union is over all primes $p\leq n$. This suggests that we use Inclusion-Exclusion to find $T$. We will pursue a related strategy which makes this computation cleaner. The key observation is that sets of the form  $\textbf{S}(a_1,a_2,a_3)$ are closed under intersection:
		\begin{equation}\label{lcm}
		\textbf{S}(a_1,a_2,a_3)\cap \textbf{S}(b_1,b_2,b_3)
			= \textbf{S}(\lcm(a_1,b_1),\lcm(a_2,b_2),\lcm(a_3,b_3))
		\end{equation}
	Furthermore, whenever we have two sets of triples $(a_1,a_2,a_3)$ and $(b_1,b_2,b_3)$ with $a_i \mid b_i$ for all $i$ we have
		$$
		\textbf{S}(b_1,b_2,b_3) \subseteq \textbf{S}(a_1,a_2,a_3)
		$$
	This implies that the triples that can be obtained by intersecting sets in \eqref{pie} form a partially ordered set under componentwise division. In the next section we will reformulate the problem in terms of this poset, then use the M\"obius Inversion Formula to get an exact formula for $T$ in terms of the $S(a_1,a_2,a_3)$. Next, we use a geometric argument to estimate the $S(a_1,a_2,a_3)$, which is the number of solutions to the Diophantine equation
		$$
		a_1x_1 + a_2x_2 + a_3x_3 = n, \quad \forall i \ x_i > 0
		$$
	Finally we combine our results to obtain \eqref{maintheorem}. 

\section{Poset Formulation}
	In this section we will define a poset which lets us express \eqref{pie} in a more convenient form. Our main tool is the M\"obius Inversion Theorem for posets. The reader may wish to refer to chapter $3$ of \cite{Stanley} for a full treatment of posets.

	Define $\textbf{A}\subset \N^3$ to be the set obtained by requiring
		\begin{itemize}
			\item
				$(1,1,1)\in \textbf{A}$
			\item
				$(p,p,1), (p,1,p), (1,p,p)\in \textbf{A}$ for all primes $p\leq n$
			\item
				$\textbf{A}$ is closed under componentwise $\lcm$. That is, for all $a,b \in \textbf{A}$ we require $\lcm(a,b) \in \textbf{A}$, where
					\begin{align*}
					\lcm(a,b) &:= (\lcm(a_1,b_1), \lcm(a_2,b_2), \lcm(a_3,b_3))
					\end{align*}
		\end{itemize}
	Define a partial order on $\textbf{A}$ by letting $a\preceq b$ if 
	and only if $a_i \mid b_i$ for all $i$. We note that $\textbf{A}$ is finite and has least element $\hat{0}:=(1,1,1)$ and greatest element $\hat{1}:= (n\#,n\#,n\#)$ where the primorial $n\#:= \prod_{p\leq n}p$

	We will relate this poset to our original problem. Let $\textbf{S}(a)$ be defined as above for every $a\in \textbf{A}$. Then for all $a,b\in \textbf{A}$ we have:
		\begin{lemma}\label{lcmlemma}
		$\textbf{S}(a)\cap \textbf{S}(b) = \textbf{S}(\lcm(a,b)) $
		\end{lemma}
			\begin{proof}
			We have $s\in \textbf{S}(a)\cap \textbf{S}(b)$ iff for all $i$ both $a_i,b_i$ divide $s_i$. This occurs iff $\lcm(a_i,b_i) \mid s_i$ which is equivalent to $s\in \textbf{S}(\lcm(a,b))$.
			\end{proof}
		\begin{lemma}\label{divlemma}
		$a \preceq b \implies \textbf{S}(b) \subseteq \textbf{S}(a)$
		\end{lemma}
			\begin{proof}
			Let $a\preceq b$ and let $s\in \textbf{S}(b)$. Then for all $i$ we have $a_i$ divides $b_i$ which divides $s_i$. Hence $s \in \textbf{S}(a)$.
			\end{proof}

	Each element $s\in \textbf{S}$ belongs to some family of sets of the form $\textbf{S}(a)$. Let $\hat{1}(s)$ be the maximal such $a$ , defined by:
		$$
		\hat{1}(s) := \lcm\left( \{ a \mid \ s\in \textbf{S}(a) \} \right)
		$$
	For each $a\in \textbf{A}$ let $\textbf{T}(a)$ be the set of all $s\in \textbf{S}$ with $\hat{1}(s) = a$:
		$$
		\textbf{T}(a) :=\{s \mid \hat{1}(s) = a \}
		$$
	Then 
		\begin{lemma}
			$\textbf{T}(\hat{0}) = \textbf{T}$
		\end{lemma}
			\begin{proof}
			We have $s\in \textbf{T}(\hat{0})$ if and only if $\hat{1}(s) = \hat{0}$. This occurs iff
				$$
				s \not \in \bigcup_{p}\left(\textbf{S}(p,p,1)\cup \textbf{S}(p,1,p) \cup \textbf{S}(1,p,p)\right)
				$$
			since otherwise we would have some $a\in \{(p,p,1),(p,1,p) ,(1,p,p) \}$ such that
				$$
				\hat{0} \prec a \preceq \hat{1}(s)
				$$
			By \eqref{pie}, see that this is equivalent to $s\in \textbf{T}$.
			\end{proof}

	With these definitions we can give an exact formula for $T$.
		\begin{theorem}
		For all $a\in \textbf{A}$, 
			$$
			T(a) = \sum_{a\preceq b} \mu_{\textbf{A}}(a,b) S(b)
			$$
		where $\mu_{\textbf{A}}$ is the M\"obius function on $\textbf{A}$
		\end{theorem}
			\begin{proof}
			Let $s\in \textbf{S}(a)$. Then $a\preceq \hat{1}(s)$. Therefore
				$$
				\textbf{S}(a) = \bigcup_{a\preceq b} \textbf{T}(b)
				$$
			Since the $\textbf{T}(b)$ are disjoint, we see that
				\begin{equation}\label{sint}
				S(a) = \sum_{a\preceq b} T(b)
				\end{equation}
			Sums of this form may be inverted by the M\"obius Inversion Formula for posets. Recall that the M\"obius function $\mu_{\textbf{A}}$ on $\textbf{A}$ is defined for pairs $a\preceq c$ in a recursive fashion:
				\begin{align*}
				\mu_{\textbf{A}}(a,a) 
					&:= 1
					\\
				\mu_{\textbf{A}}(a,c)
					&:= - \sum_{a\preceq b\prec c} \mu_{\textbf{A}}(a,b)
				\end{align*}
			This guarantees that
				$$
				\sum_{a\preceq b\preceq c} \mu_{\textbf{A}}(a,c)
					= \delta_{ac}
				$$
			Hence
				\begin{align*}
				T(a)
					&= \sum_{a\leq c} T(c) \delta_{ac}
					\\
					&= \sum_{a\leq c} T(c) \sum_{a\preceq b\preceq c} \mu_{\textbf{A}}(a,b)
				\end{align*}
			Interchanging the order of summation and using \eqref{sint} yields the desired result:
				\begin{align*}
				T(a)
					&= \sum_{a\leq b} \mu_{\textbf{A}}(a,b) \sum_{b\preceq c} T(c)
					\\
					&=  \sum_{a\leq b} \mu_{\textbf{A}}(a,b) S(b)
				\end{align*}
			\end{proof}
	As a consequence, we see that
		\begin{corollary}\label{explicitmobius}
		\begin{equation*}
		T = T(\hat{0}) = \sum_{a\in \textbf{A}} \mu_{\textbf{A}}(\hat{0},a) S(a)
		\end{equation*}
		\end{corollary}
	In the next section we use the properties of $\textbf{A}$ to derive this M\"obius function explicitly.

\section{Computing the M\"obius Function}
	To compute the M\"obius Function we investigate the structure of $\textbf{A}$. This can be characterized as follows. For each prime $p\leq n$ let $\textbf{I}_p\subseteq \textbf{A}$ be the subposet with elements (see Figure \ref{Ip})
		$$
		\textbf{I}_p = \{\hat{0}, (p,p,1), (p,1,p), (1,p,p), (p,p,p) \}
		$$

	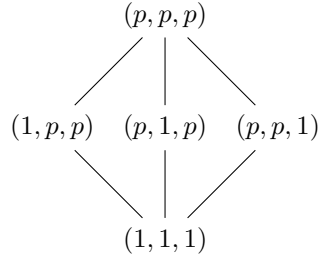
\begin{figure}
		\centering
		\begin{tikzpicture}[scale=1.5]
		  \node (one) at (0,2) {$(p,p,p)$};
		  \node (a) at (-1,1) {$(1,p,p)$};
		  \node (b) at (0,1) {$(p,1,p)$};
		  \node (c) at (1,1) {$(p,p,1)$};
		  \node (zero) at (0,0) {$(1,1,1)$};
		  \draw (zero) -- (a) -- (one) -- (b) -- (zero) -- (c) -- (one);
		\end{tikzpicture}
		\caption{The Hasse diagram for $I_p$}
		\label{Ip}
	\end{figure}

	Let $\textbf{I}$ be the Cartesian product of the $\textbf{I}_p$, which we will denote by
		$$
		\textbf{I}:= \otimes_{p\leq n} \textbf{I}_p
		$$
	Then
		\begin{lemma}\label{posetiso}
		$\textbf{A}\cong \textbf{I}$
		\end{lemma}
			\begin{proof}
			This amounts to "factoring" every $a\in \textbf{A}$ componentwise. Specifically, for any prime $p\leq n$ let $\phi_p(a)$ be the triple whose $i$th component is given by
				$$
				\phi_p(a)_i :=
				\begin{cases}
				p \quad : \quad p \mid a_i
				\\
				1 \quad : \quad p \nmid a_i
				\end{cases}
				$$
			then it's easy to see from the definition of $\textbf{A}$ that $\phi_p(a) \in \textbf{I}_p$. By collecting all the $\phi_p$ we obtain the map $\phi:\textbf{A}\to \textbf{I}$ given by
				$$
				\phi(a) := \otimes_{p\leq n}\phi_p(a)
				$$
			This map is invertible, with inverse taking $x\in \textbf{I}$ with $x = \otimes_{p\leq n} x_p$ to the triple whose $i$th component is obtained by componentwise multiplication
				$$
				\phi^{-1}(x)_i := \prod_{p\leq n}(x_p)_i
				$$
			Hence $\phi$ is a bijection. To see that $\phi$ is order preserving, note that for $a,b\in \textbf{A}$ we have $a\preceq b$ if and only if $a_i\mid b_i$ for every $i$. This occurs if and only if for each prime $p\mid a_i$ implies $p\mid b_i$, which is equivalent to $\phi_p(a)\preceq \phi_p(b)$ for all $p$, or $\phi(a)\preceq \phi(b)$.
			\end{proof}
	\begin{figure}[ht]
		\centering
		\begin{tikzpicture}[scale=1.5]
		  \node (one) at (0,2) {$2$};
		  \node (a) at (-1,1) {$-1$};
		  \node (b) at (0,1) {$-1$};
		  \node (c) at (1,1) {$-1$};
		  \node (zero) at (0,0) {$1$};
		  \draw (zero) -- (a) -- (one) -- (b) -- (zero) -- (c) -- (one);
		\end{tikzpicture}
		\caption{$\mu_{I_p}(\hat{0},x)$ for all $x\in \textbf{I}_p$}
		\label{muIp}
	\end{figure}
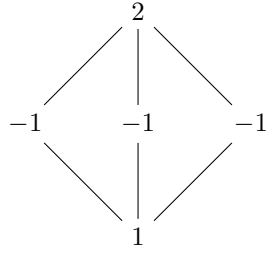
	We can use this to compute M\"obius function on $\textbf{A}$. Let $\phi$ be the poset isomorphism in \eqref{posetiso}. Since Isomorphic Posets have the same M\"obius function we see that
		\begin{align*}
		\mu_{\textbf{A}}(a,b) 
			&= \mu_{\textbf{I}}(\phi(a),\phi(b))
		\end{align*}
	Since the M\"obius function of a product is the product of the M\"obius functions we obtain
		\begin{align}
		\mu_{\textbf{I}}(\phi(a),\phi(b))
			&= \prod_{p\leq n} \mu_{\textbf{I}_p}(\phi_p(a),\phi_p(b))
			\label{mobiusproduct}
		\end{align}
	The $\mu_{\textbf{I}_p}$ can be determined by inspection. We are interested in evaluating the sum \eqref{explicitmobius}, so we need 
		$$
		\mu_{\textbf{I}_p}(\phi_p(\hat{0}), \phi_p(a)) = \mu_{\textbf{I}_p}(\hat{0}, \phi_p(a))
		$$
	for all $a\in \textbf{A}$. We will express this in a way which will be convenient later on. For all primes $p$ and $a\in \textbf{A}$ let $\kappa_p(a)$ be the number of components $a_i$ divisible by $p$. Then for all $x\in \textbf{I}_p$ we have (see Figure \ref{muIp})
		$$
		\mu_{\textbf{I}_p}(\hat{0}, x)
			= (-1)^{\kappa_p(x)-1}(\kappa_p(x)-1)
		$$
	Then since $\kappa_p(a) = \kappa_p(\phi_p(a))$ we can use \eqref{mobiusproduct} to obtain
		\begin{align*}
		\mu_{\textbf{A}}(\hat{0},a)
			&= \prod_{p\leq n} \mu_{\textbf{I}_p}(\hat{0},\phi_p(a))
			\\
			&= \prod_{p\leq n} (-1)^{\kappa_p(\phi_p(a))-1}(\kappa_p(\phi_p(a))-1)
			\\
			&= \prod_{p\leq n} (-1)^{\kappa_p(a)-1}(\kappa_p(a)-1)
		\end{align*}
	Plugging this into \eqref{explicitmobius} we obtain the following form for $\textbf{T}$:
		\begin{theorem}\label{explicit}
		\begin{equation*}
		T = \sum_{a\in \textbf{A}} \prod_{p\leq n} (-1)^{\kappa_p(a)-1}(\kappa_p(a)-1) S(a)
		\end{equation*}
		\end{theorem}
	Although there are $|\textbf{A}| = |\textbf{I}| = 4^{\pi(n)}$ terms in this sum so to make use of this we need to estimate the $S(a)$ carefully. In the next section we derive a bound for $S(a)$ for all $a\in \textbf{A}$.

\section{Estimating $S(a)$}
	We will estimate $S(a)$ for all $a\in \textbf{A}$. Recall that $s\in \textbf{S}(a)$ if and only if $a_i\mid s_i$ for all $i$. Writing each $s_i = a_ix_i$ we see that every $s$ corresponds to a unique positive solution to the Diophantine equation
		\begin{equation}\label{ndiophantine}
		a_1x_1 + a_2x_2 + a_3x_3 = n
		\end{equation}
	Let $d(a):=\gcd(a_1,a_2,a_3)$. Then \eqref{ndiophantine} has integer solutions if and only if $d(a)\mid n$. We will determine the general solution, then use geometry to estimate $S(a)$. 

	We introduce some notation. For the rest of this section let $a\in \textbf{A}$ such that $d(a)\mid n$, and for convenience let $d:=d(a)$.
	For $i\not=j\not=k$ let
		$$
		g_i = \gcd\left( \frac{a_j}{d}, \frac{a_k}{d} \right)
		$$
	Since the $a_i$ are squarefree we see that $g_i$ is the product of all primes dividing both $a_j,a_k$ but not $a_i$ or $d$. It follows that
		\begin{lemma} \label{gscoprime}
			\begin{align*}
			\gcd(g_i,g_j) &= \gcd(g_i,d) = 1
			\\
			a_i &= dg_jg_k
			\end{align*}
		\end{lemma}
	Let $m:=n/d$. Then \eqref{ndiophantine} becomes
		\begin{align}\label{mdiophantine}
		g_2g_3x_1 + g_1g_3x_2 + g_1g_2x_3 = m
		\end{align}
	For any $m\in \Z$ let $\textbf{X}_m\subset \Z^3$ be the set of integer solutions to \eqref{mdiophantine}. Since we already removed the $\gcd$ of the coefficients, we know that $\textbf{X}_m$ is nonempty. For for any $v\in \textbf{X}_m$ we may write
		\begin{equation}
		\textbf{X}_m = v + \textbf{X}_0 \label{coset}
		\end{equation}
	Geometrically, $\textbf{X}_m$ is a translation of the lattice $\textbf{X}_0$; we will investigate the latter. For any $x\in \textbf{X}_0$ we have
		\begin{align}
		g_2g_3x_1 + g_1g_3x_2 + g_1g_2x_3 
			&= 0 \label{x0}
			\\
		g_2g_3x_1 
			&=-g_1(g_3x_2 + g_2x_3) \label{x0divis}
		\end{align}
	The right side of \eqref{x0divis} is divisible by $g_1$ so the left side is as well. By \eqref{gscoprime} we know that the $g_i$ are pairwise coprime, so we must have $g_1\mid x_1$. The same reasoning shows that $g_2,g_3$ divide $x_2,x_3$ respectively. Factoring these out to write $x_i = g_i y_i$ for some $y_i$, we have
		\begin{align}
		g_2g_3(g_1y_1) + g_1g_3(g_2y_2) + g_1g_2(g_3y_3) 
			&= 0 \nonumber
			\\
		(g_1g_2g_3)(y_1+y_2+y_3 )
			&= 0 \nonumber
			\\
		y_1+y_2+y_3
			&= 0 \label{ys}
		\end{align}
	Thus each $x\in \textbf{X}_0$ corresponds to a unique solution to \eqref{ys}. It's easy to see that the reverse holds as well; given any solution to \eqref{ys} we get a unique element of $\textbf{X}_0$ by the map $y_i \mapsto g_iy_i$. We can write the solution set to \eqref{ys} as
		$$
		\begin{pmatrix}
		1 & 0
		\\
		0 & 1
		\\
		-1 & -1
		\end{pmatrix} \Z^2
		$$
	Thus we may express $\textbf{X}_m$ as
		$$
		\textbf{X}_m = v + \begin{pmatrix}
		g_1 & 0
		\\
		0 & g_2
		\\
		-g_3 & -g_3
		\end{pmatrix} \Z^2
		$$
	Now we are ready to estimate $S(a)$. Let $\textbf{P}_m\subset \R^3$ be the plane determined by $\textbf{X}_m$. Let $\Delta_m\subset \textbf{P}_m$ be the interior of the triangle with vertices
		$$
		(m/(g_2g_3),0,0),(0,m/(g_1g_3),0),(0,0,m/(g_1g_2))
		$$
	Then
		$$
		S(a) = |\Delta_m \cap \textbf{X}_m|
		$$
	Let $F$ be the scaled projection map given by 
		\begin{align*}
		F&:=
		\begin{pmatrix}
		1/g_1 & 0 & 0
		\\
		0 & 1/g_2 & 0
		\end{pmatrix}
		\end{align*}
	Then since the restriction of $F$ to $\textbf{P}_m$ is a bijection to $\R^2$, we see that
		$$
		S(a) = |F \Delta_m \cap F\textbf{X}_m|
		$$
	Applying $F$ gives
		$$
		F \textbf{X}_m
		=
		F v + 
		\Z^2
		$$
	So $F \textbf{X}_m$ is a translate of $\Z^2$. On the other hand, $F \Delta_m$ is $m/(g_1g_2g_3)\Delta$ where $\Delta$ is the interior of the triangle with vertices
		$$
		(1,0),(0,1),(0,0)
		$$
	For all $\lambda > 0 $ let $L(\lambda)$ be the number of points of $F \textbf{X}_m$ within $\lambda \Delta$, so that
		$$
		L(\lambda):= |\lambda \Delta \cap F \textbf{X}_m|
		$$
	Then $S(a) = L(m/(g_1g_2g_3))$. The function $L(\lambda)$ is determined by the coset $F\textbf{X}_m = F v + \Z^2$, but should be close to the area $\lambda^2/2$. We will show that
		\begin{lemma}
		There exists some integer $t\geq 0$ with $|\lambda-t|\leq 2$ such that
			$$
			L(\lambda) = \binom{t}{2}
			$$
		\end{lemma}
			\begin{proof}
			Let $\textbf{Q}_m$ be the set of points of $F\textbf{X}_m$ which have positive coordinates.	Let $q=(q_1,q_2)$ be the unique element of $\textbf{Q}_m$ which lies in the half-open square $(0,1]^2$, so $\textbf{Q}_m = q + \N^2$. Then
				\begin{align*}
				L(\lambda) 
					&= |\lambda \Delta \cap \textbf{Q}_m|
				\end{align*}
			Let $\lambda_q:= q_1 + q_2$. If $\lambda \leq \lambda_q$ we have $L(\lambda) = 0 = \binom{1}{2}$, so the lemma holds in this case. Now suppose that $\lambda > \lambda_q$, and let $s$ be the largest natural number for which $\lambda_q + s < \lambda$. Then for each natural number $0\leq s' \leq s$, there are $s'+1$ points of $\lambda \Delta \cap \textbf{Q}_m$ on the line 
				$$
				z_1 + z_2 = \lambda_q + s'
				$$
			Moreover every such point is on a line of this form. Thus
				\begin{align*}
				L(\lambda)
					&= \sum_{0\leq s' \leq s} s' + 1
					\\
					&= \binom{s+2}{2}
				\end{align*}
			To complete the proof we set $t:=s+2$ and use the fact that 
				$$
				0 < \lambda_q \leq \lambda-s \leq \lambda_q + 1\leq 3
				$$
			to obtain $|\lambda - t| \leq 2$.
			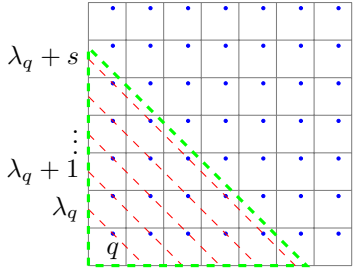
\begin{figure}[ht]
				\centering
				\begin{tikzpicture}[scale=.5]
				\tikzmath{
				\maxcoord = 7;
				\maxlincoord = \maxcoord -3;
				\lam = \maxcoord - 2 + .6;
				\vx = -.35;
				\vy = -.15;
				\slam = 2+\vx+\vy;
				}
				\draw[step=1, gray, thin] (0,0) grid (\maxcoord,\maxcoord);
				\foreach \x in {1,...,\maxcoord}
					\foreach \y in {1,...,\maxcoord}
						{
						\filldraw[blue] (\x+\vx,\y+\vy) circle [radius=1.2pt];
						}
				\foreach \x in {0,...,\maxlincoord}
					{
					\draw[red,dashed] (0,\slam+\x) -- (\slam+\x,0);
					}
				\draw (0,\slam) node[anchor=east] {$\lambda_q$};
				\draw (0,\slam+1) node[anchor=east] {$\lambda_q+1$};
				\draw (0,\slam+2) node[anchor=east] {$\vdots$};
				\draw (0,\slam+4) node[anchor=east] {$\lambda_q+s$};
				\draw[green,dashed,very thick] (0,0) -- (0,\slam+4.3) -- (\slam+4.3,0) -- (0,0);
				\draw (1+\vx,1+\vy) node[anchor=north] {$q$};
				\end{tikzpicture}
				\caption{$\lambda \Delta$ is depicted in green,
				the lattice $\textbf{Q}_m$ in blue. 
				}
				\label{lambdatriangles}
			\end{figure}
			\end{proof}
	Let
		$$
		l(a) := g_1g_2g_3
		$$
	Then we can summarize the results of this section in the following way.
		\begin{theorem}\label{Sestimate}
		Let $a\in \textbf{A}$ and let $d(a),l(a)$ be defined as above. If $d(a)\nmid n$ then $S(a) = 0$. If $d(a)\mid n$, then there exists some real number $|r(a)|\leq 2$ such that
		$$
		S(a) = \frac{1}{2}\left( \frac{n}{d(a)l(a)} + r(a) \right)\left( \frac{n}{d(a)l(a)} + r(a) - 1 \right)
		$$
		\end{theorem}
		
\section{Computing $T$}
	We will now evaluate \ref{explicit}. We have
		\begin{equation*}
		T = \sum_{a\in \textbf{A}} \prod_{p\leq n} (-1)^{\kappa_p(a)-1}(\kappa_p(a)-1) S(a)
		\end{equation*}
	Using the fact that $S(a)=0$ unless $d(a)\mid n$ and substituting the result of Theorem \ref{Sestimate} we have
		\begin{align*}
		T 
			&= \sum_{\substack{a\in \textbf{A}\\ d(a)\mid n}} \prod_{p\leq n} (-1)^{\kappa_p(a)-1}(\kappa_p(a)-1) S(a)
			\\
			&= \sum_{\substack{a\in \textbf{A}\\ d(a)\mid n}} \prod_{p\leq n} (-1)^{\kappa_p(a)-1}(\kappa_p(a)-1) \frac{1}{2}\left( \frac{n}{d(a)l(a)} + r(a) \right)\left( \frac{n}{d(a)l(a)} + r(a) - 1 \right)
			\\
			&= M + E_1 + E_2
		\end{align*}
	Where
		\begin{align}
		M
			&:= \frac{n^2}{2}\sum_{\substack{a\in \textbf{A}\\ d(a)\mid n}} \prod_{p\leq n} (-1)^{\kappa_p(a)-1}(\kappa_p(a)-1)\left( \frac{1}{d(a)l(a)}\right)^2
			\label{Mdef}
			\\
		E_1
			&:= \frac{n}{2}\sum_{\substack{a\in \textbf{A}\\ d(a)\mid n}} \prod_{p\leq n} (-1)^{\kappa_p(a)-1}(\kappa_p(a)-1)\frac{2r(a)-1}{d(a)l(a)}
			\label{E1def}
			\\
		E_2
			&:=\sum_{\substack{a\in \textbf{A}\\ d(a)\mid n}} \prod_{p\leq n} (-1)^{\kappa_p(a)-1}(\kappa_p(a)-1)r(a)(r(a)-1)
			\label{E2def}
		\end{align}
	We will evaluate each of these sums separately, beginning with the main term $M$.

	\subsection{The Main Term}
	We have
		\begin{align}
		M
			&:= \frac{n^2}{2}\sum_{\substack{a\in \textbf{A}\\ d(a)\mid n}} \prod_{p\leq n} (-1)^{\kappa_p(a)-1}(\kappa_p(a)-1)\left( \frac{1}{d(a)l(a)}\right)^2
			\label{M}
		\end{align}
	Recall that $d(a)$ is the product of all primes which divide all three of the $a_i$ while $l(a)$ is the product of the primes dividing exactly two of the $a_i$. This lets us factor each term in the sum as
		\begin{align}
		\prod_{p\leq n} (-1)^{\kappa_p(a)-1}(\kappa_p(a)-1)\left( \frac{1}{d(a)l(a)}\right)^2
			&= \left(\prod_{p \mid d(a)}\frac{2}{p^2}\right) \left(\prod_{q \mid l(a)}-\frac{1}{q^2}\right)
			\label{Mtermprod}
		\end{align}
	Let $\textbf{D}$ be the set of squarefree numbers whose prime factors are all less than or equal to $p$. Then for any $a\in \textbf{A}$ we see that $d(a),l(a)\in \textbf{D}$. Moreover every term in \eqref{M} corresponds to a paird $(d,l)\in \textbf{D}^2$ such that $d\mid n$ and $\gcd(d,l) = 1$.	Using this and substituting \eqref{Mtermprod} into \eqref{M} we obtain
		\begin{align*}
		M
			&= \frac{n^2}{2}
			\sum_{\substack{(d,l) \in \textbf{D}^2 \\ d \mid n \\ \gcd(d,l) = 1}} 
			\left(\prod_{p \mid d}\frac{2}{p^2}\right) 
			\left(\prod_{q \mid l}-\frac{1}{q^2}\right)
			\sum_{\substack{a\in \textbf{A} \\ d(a)=d\\ l(a) = l}} 
			1
		\end{align*}
	If $d(a) = d$ and $l(a) = l$ then every prime dividing $d$ divides all of the $a_i$, while primes dividing $l$ divide exactly two of the $a_i$. There are $3$ ways that this can happen for each prime $p \mid l$, hence for each term in the sum we obtain
		$$
		\sum_{\substack{a\in \textbf{A} \\ d(a)=d\\ l(a) = l}} 
				1
			= \prod_{q \mid l}3
		$$
	Therefore
		\begin{align}
		M
			&= \frac{n^2}{2}
			\sum_{\substack{(d,l) \in \textbf{D}^2 \\ d \mid n \\ \gcd(d,l) = 1}} 
			\left(\prod_{p \mid d}\frac{2}{p^2}\right) 
			\left(\prod_{q \mid l}-\frac{3}{q^2}\right)
			\nonumber
			\\
			&= \frac{n^2}{2}
			\sum_{\substack{d\in \textbf{D} \\ d \mid n}} 
			\left(\prod_{p \mid d}\frac{2}{p^2}\right) 
			\sum_{\substack{l\in \textbf{D} \\ \gcd(d,l) = 1}} 
			\left(\prod_{q \mid l}-\frac{3}{q^2}\right)
			\label{Msplit}
		\end{align}
	The rightmost sum can be factored:
		\begin{align*}
		\sum_{\substack{l\in \textbf{D} \\ \gcd(d,l) = 1}} 
		\left(\prod_{q \mid l}-\frac{3}{q^2}\right)
			&= \prod_{\substack{q\leq n \\ q\nmid d}}\left(1-\frac{3}{q^2}\right)
			\\
			&=
			\prod_{\substack{q' \mid n \\ q'\nmid d}}\left(1-\frac{3}{q'^2}\right)
			\prod_{\substack{q\leq n \\ q\nmid n}}\left(1-\frac{3}{q^2}\right)
		\end{align*}
	Plugging into \eqref{Msplit} and pulling out the common factor $\prod_{\substack{q\leq n \\ q\nmid n}}\left(1-\frac{3}{q^2}\right)$ we have 
		\begin{align*}
		M
		&=
		\frac{n^2}{2}
		\sum_{\substack{d\in \textbf{D} \\ d \mid n}} 
		\left(\prod_{p \mid d}\frac{2}{p^2}\right) 
		\prod_{\substack{q' \mid n \\ q'\nmid d}}\left(1-\frac{3}{q'^2}\right)
		\prod_{\substack{q\leq n \\ q\nmid n}}\left(1-\frac{3}{q^2}\right)
		\\
		&=
		\prod_{\substack{q\leq n \\ q\nmid n}}\left(1-\frac{3}{q^2}\right) 
		\frac{n^2}{2}
		\sum_{\substack{d\in \textbf{D} \\ d \mid n}} 
		\left(\prod_{p \mid d}\frac{2}{p^2}\right) 
		\prod_{\substack{q' \mid n \\ q'\nmid d}}\left(1-\frac{3}{q'^2}\right)
		\end{align*}
	Since $d\mid n$ we see that this final sum can also be factored in terms of the prime divisors of $n$, giving
		\begin{align*}
		\sum_{\substack{d\in \textbf{D} \\ d \mid n}} 
			\left(
				\prod_{p \mid d}\frac{2}{p^2}
			\right) 
			\prod_{\substack{q' \mid n \\ q'\nmid d}}\left(1-\frac{3}{q'^2}\right)
		&=
		\prod_{\substack{p \mid n}}\left(1 - \frac{3}{p^2}+\frac{2}{p^2} \right)
		\\
		&=
		\prod_{\substack{p \mid n}}\left(1 - \frac{1}{p^2}\right)
		\end{align*}
	Thus
		$$
		M = \frac{n^2}{2}\prod_{\substack{p\leq n \\ p\nmid n}}\left(1-\frac{3}{p^2}\right)\prod_{\substack{p \mid n}}\left(1 - \frac{1}{p^2}\right)
		$$
	To get the main term of Theorem \ref{maintheorem}, notice that
		\begin{align*}
		\prod_{p > n} \left(1-\frac{3}{p^2}\right)M
			&= f(n)\frac{n^2}{2}
		\end{align*}
	This product is 
		$$
		\prod_{p > n} \left(1-\frac{3}{p^2}\right)
			= 1 + O(1/n)
		$$
	Hence
		\begin{align*}
		M
			&= f(n)\frac{n^2}{2} + O(n)
		\end{align*}

	\subsection{First Error Term}
	Now we turn to the first error term, $E_1$. Using the fact that $|r(a)|\leq 2$ we have
		\begin{align}
		|E_1|
			&= \Big|\frac{n}{2}\sum_{\substack{a\in \textbf{A}\\ d(a)\mid n}} \prod_{p\leq n} (-1)^{\kappa_p(a)-1}(\kappa_p(a)-1)\frac{2r(a)-1}{d(a)l(a)} \Big|
			\nonumber
			\\
			&\leq \frac{5n}{2}\sum_{\substack{a\in \textbf{A}\\ d(a)\mid n}} \prod_{p\leq n} |\kappa_p(a)-1|\frac{1}{d(a)l(a)} 
			\label{E1abs}
		\end{align}
	This sum resembles the sum for $M$, and can be factored in much the same way. We omit the details, but it is easily shown that
		\begin{align*}
		\sum_{\substack{a\in \textbf{A}\\ d(a)\mid n}} \prod_{p\leq n} |\kappa_p(a)-1|\frac{1}{d(a)l(a)} 
			&=
			\prod_{\substack{q\leq n \\ q\nmid n}}\left(1+\frac{3}{q}\right)\prod_{\substack{p \mid n}}\left(1 + \frac{5}{p}\right)
			\\
			&=O(\log^5 n)
		\end{align*}
	This last bound follows from Merten's Theorem. Hence
		$$
		|E_1| = O\left( n\log^5 n \right)
		$$

	\subsection{Second Error Term}
	Now we estimate the remaining error. Let $\textbf{A}_0\subset \textbf{A}$ be the set of all $a$ for which $S(a)=0$ and let $\textbf{A}_1 = \textbf{A}\setminus \textbf{A}_0$. Using this we can break $E_2$ into two sums
		\begin{align*}
		E_2
			&= F_0 + F_1
		\end{align*}
	Where each $F_i$ is the sum of terms coming from $\textbf{A}_i$
		\begin{align*}
		F_i
			&:=
			\sum_{\substack{a\in \textbf{A}_i\\ d(a)\mid n}} \prod_{p\leq n} (-1)^{\kappa_p(a)-1}(\kappa_p(a)-1)r(a)(r(a)-1)
		\end{align*}
	We'll estimate $F_0$ first.	If $a$ appears in $F_0$ then $S(a) = 0$ and $d(a) \mid n$ so by Theorem \ref{Sestimate} we have
		$$
		\left( \frac{n}{d(a)l(a)} + r(a) \right)\left( \frac{n}{d(a)l(a)} + r(a) - 1 \right) = 0
		$$
	We see that one of $r(a), r(a)-1$ must equal $-\frac{n}{d(a)l(a)}$. The other is bounded above by $3$, so we see that
		$$
		|r(a)(r(a)-1)|\leq 3 \frac{n}{d(a)l(a)}
		$$
	Thus
		\begin{align*}
		|F_0 |
			&\leq \sum_{\substack{a\in \textbf{A}_0\\ d(a)\mid n}} \prod_{p\leq n} |\kappa_p(a)-1||r(a)(r(a)-1)|
			\\
			&\leq 3\sum_{\substack{a\in \textbf{A}_0\\ d(a)\mid n}} \prod_{p\leq n} |\kappa_p(a)-1|\frac{n}{d(a)l(a)}
			\\
			&\leq 3\sum_{\substack{a\in \textbf{A}\\ d(a)\mid n}} \prod_{p\leq n} |\kappa_p(a)-1|\frac{n}{d(a)l(a)}
		\end{align*}
	This is the same sum used to bound $E_1$, so we see that
		$$
		|F_0| = O(n\log^5n)
		$$
	The bulk of the error comes from $|F_1|$. Using the bound $|r(a)(r(a)-1)|\leq 6$ we have
		\begin{align}
		|F_1|
			&\leq 6\sum_{\substack{a\in \textbf{A}_1\\ d(a)\mid n}} \prod_{p\leq n}|\kappa_p(a)-1|
			\label{F1basicbound}
		\end{align}
	For a positive natural number $x$ let $\omega(x)$ be the number of distinct prime factors. Then
		\begin{align*}
		\prod_{p\leq n}|\kappa_p(a)-1|
		&= 2^{\omega(d(a))}
		\\
		&\leq 2^{\omega(n)}
		\\
		&\leq n^{O(1/\log\log n)}
		\end{align*}
	The last inequality follows from using the Prime Number Theorem to obtain the bound $\omega(x) = O\left( \frac{\log x}{\log \log x} \right)$. Plugging this into \eqref{F1basicbound} we obtain
		\begin{align}
		|F_1|
			&\leq n^{O(1/\log\log n)} |\textbf{A}_1|
		\end{align}
	We give a rough bound for $|\textbf{A}_1|$. If $a\in \textbf{A}_1$, then $a_1+a_2+a_3 \leq n$ which implies that $a_1a_2a_3\leq (n/3)^{3}$. This implies that $l(a) \leq (n/3)^{3/2}$. There are at most $3^{\omega(l)}$ elements $a$ for which $l(a) = l$. Hence
		\begin{align*}
		|\textbf{A}_1| 
			&\leq \sum_{l \leq (n/3)^{3/2}}3^{\omega(l)}
			\\
			&\leq \sum_{l \leq (n/3)^{3/2}}3^{\omega(l)}
			\\
			&\leq \sum_{l \leq (n/3)^{3/2}}l^{O(1/\log\log n)}
			\\
			&\leq n^{3/2 + O(1/\log\log n)}
		\end{align*}

\bibliography{newreferences}{}
\bibliographystyle{plain}

\end{document}